\documentclass[12pt]{amsart}

 \usepackage{graphicx}
 \usepackage{amscd, color}

\usepackage{marginnote}

 \def\al{\alpha}
 \def\be{\beta}
  
 \def\de{\delta}

 \def\ga{\gamma}

 \def\la{\lambda}

 \def\si{\sigma}

 \def\EE{{\mathbf E}}
 \def\VV{{\mathbf V}}
 \def\XX {{\mathbf X}}
 \def\EN{{\mathcal E}}

 \def\R{{\mathbb R}}

  \def\A{{\mathcal  A}}

 \def\ov{\overline}
 \def\un{\underline}

 \def\oo{\mathrm o}
 \def\tt{{\mathrm t}}

\def\um{u_{\mathrm{max}}}

  \renewcommand{\proofname}{{\bf Proof:}}

 \theoremstyle{plain}

  \swapnumbers

 \newtheorem{Thm}{Theorem}[section]
 
 \newtheorem{Lemma}[Thm]{\bf Lemma}
 
 \newtheorem{Corollary}[Thm]{\bf Corollary}
 \newtheorem{Theorem}[Thm]{\bf Theorem}
 \newtheorem{Proposition}[Thm]{\bf Proposition}

 \theoremstyle{definition}
 \newtheorem{Definition}[Thm]{\bf Definition}

 \theoremstyle{remark}
 \newtheorem{Remark}[Thm]{\bf Remark}

 \newtheoremstyle{Cl}
  {5pt}
  {3pt}
  {\sl}
  {}
  {\it}
  {:}
  {.5em}
  {}

 \theoremstyle{Cl}

 \def\begincproof{
                  \renewcommand{\proofname}{\it Proof:}
                  \begin{proof}
                 }

 \def\endcproof{
                \renewcommand{\qedsymbol}{$\diamondsuit$}
                \end{proof}
                \renewcommand{\qedsymbol}{\openbox}
                \renewcommand{\proofname}{\bf Proof:}
               }

 \renewcommand{\proofname}{{\bf Proof:}}

\title[Discounted H-J equations on networks]
{Discounted Hamilton-Jacobi equations on networks  and asymptotic
analysis}

\author{Marco Pozza}
\address{Dipartimento di Matematica, Sapienza Universit\`a  di Roma, Italy.}
\email{pozza@mat.uniroma1.it}

\author{Antonio Siconolfi}
\address{Dipartimento di Matematica, Sapienza Universit\`a  di Roma, Italy.}
\email{siconolfi@mat.uniroma1.it}


\begin{document}
\maketitle

\begin{abstract} We study discounted Hamilton--Jacobi equations on networks, without putting any restriction on their geometry. Assuming the
Hamiltonians are  continuous and coercive, we establish a comparison principle and provide representation formulae for solutions.  We follow
the approach introduced in  \cite{SiconolfiSorrentino}, namely we associate to the differential problem on the network, a discrete functional
equation on an abstract underlying graph. We perform some qualitative analysis and single out a distinguished subset of vertices, called
$\lambda$--Aubry set,  which shares some properties of the Aubry set for Eikonal equations on compact manifolds.  We finally study the
asymptotic behavior of solutions and $\lambda$--Aubry sets as the discount factor $\lambda$ becomes infinitesimal.
\end{abstract}

\section{Introduction}  We are concerned with discounted Hamilton--Jacobi equations on networks. We establish a comparison principle, provide
representation formulae for solutions  and perform some qualitative analysis. We emphasize that our results apply without any restriction on
the geometry of the network. In particular multiple arcs connecti ng a given pair of vertices are allowed as well as loops or multiple loops
based on a single vertex.

 Given a finite family of  Hamiltonians $H_\ga$  defined on $[0,1] \times \R$, indexed by a parameter $\ga$,  we consider the corresponding
 discounted equations
\begin{equation}\label{start}
 \la\, u + H_\ga(s,u')=0 \qquad\hbox{in $(0,1)$},
\end{equation}
with discount factor  $\la$   independent of $\ga$.  The $H_\ga$ are assumed continuous in both arguments and coercive in the momentum
variable, no convexity is required, see assumptions {\bf (H1)}, {\bf (H2)} in Section \ref{setting}.

Since  the Hamiltonians are unrelated and no boundary conditions are specified,  these equations possess infinite viscosity solutions, when
separately considered.

A  sort of geometric coupling  is  provided  by setting each equation  on an arc of a given network $\Gamma$ immersed in $\R^N$, and
combining them  with  additional conditions at the  vertices, namely  at  the junction points of different arcs.  In other terms  boundary
conditions are introduced  in  correspondence to endpoints $0$ and $1$ of the parametrization.  The subtle point however  is that these
conditions are not  required  in the same way at all vertices for supersolutions,  but are given taking into account the geometry of the
network, as made precise in   Definition \ref{def:HJsol} {\bf iii)}.

The aim  is to uniquely select  distinguished solutions  of all equations which piece together continuously   at vertices, in other terms  to
uniquely determine a solution of  the differential problem on $\Gamma$. Namely a continuous functions $u: \Gamma \longrightarrow \R$
satisfying
\[ \la \,  u \circ \ga  + H_\ga(s,( u \circ \ga)')= 0 \qquad\hbox{ in $(0,1)$}\]
in the viscosity sense for any arc $\ga$, plus   vertex conditions.  Following \cite{LionsSouganidis},  \cite{SiconolfiSorrentino}  we
consider state constraint type boundary conditions which correspond  to, at least on the arcs where these boundary conditions apply,  so called maximal solutions of  \eqref{start}. By this we mean that fixing a number $\alpha$, and considering the family $S_\alpha$ of all solutions taking the value
$\alpha$ at $0$, the element of  $S_\alpha$ which also satisfies the state constraint boundary condition at $1$ is  maximal  in $S_\alpha$.

We attack the problem  through the approach introduced in \cite{SiconolfiSorrentino} for the  Eikonal case. Namely we associate to the above
described  problem  on $\Gamma$ a discrete equation  defined on an underlying  graph, which has   the same vertices of $\Gamma$ and  edges
corresponding to the arcs of $\Gamma$.

The two problems  are related  by the fact that the trace on the vertices of a solution of the continuous equation solves the discrete one,
and conversely any solution of the discrete equation can be uniquely extended, from vertices to the whole network, to a solution of the HJ
discounted equation. See \eqref{HJdis} and Theorem \ref{theo:ponte}, Proposition \ref{pontebis} in Section \ref{statgraph}.

We can therefore prove existence and comparison results for the discrete equation and then transfer it to the differential problem on the
network. The advantage of this procedure  is twofold. The   comparison principles are obtained through  simple combinatorial techniques
bypassing Crandall--Lions doubling variables  method,  see Theorem \ref{sol!}.   In addition,  explicit representation formulae for solutions
can be provided,  see \eqref{soldis}, which   makes  possible a  qualitative analysis, using   a suitable functional defined  on the paths of
the graph, see  Definition \ref{functional}.

In this way we can single out a special subset of vertices (and edges), called $\lambda$--Aubry set, which  shares some properties of  the
Aubry set for Eikonal equations on networks with convex or quasiconvex  Hamiltonians, see   \cite{SiconolfiSorrentino}.  A similar entity has
been found  for  discounted  equations with regular Hamiltonians, (contact Hamiltonians)  on  compact manifolds  in  \cite{MaroSorrentino},
\cite{WangYan} via dynamical techniques.

Assuming the $H_\ga$ convex, we study the link, as $\lambda$ becomes infinitesimal, of $\lambda$--Aubry set  with the Aubry set of the
corresponding Eikonal equation ($\lambda =0$). In particular  we show, see Proposition \ref{superpoor}, that, for $\lambda$ suitably small,
the $\lambda$--Aubry sets are contained in the Aubry set for the Eikonal equation. This should be compared with the convergence result
established in \cite{MaroSorrentino}.

The paper is organized as follows:  The problem under investigation is presented  in Section \ref{setting}  together with the assumptions on
the Hamiltonians and  the main related definitions. In Section \ref{boundary} we summarize the relevant  properties of one--dimensional
discounted HJ equations posed on an interval.

 In Section \ref{statgraph} we introduce the discrete equation,  prove the link with the differential problem on the network, and establish a
 comparison principle.  Section \ref{discrete}   is devoted to the definition of a functional on the paths of the graph, which will play a
 major role in the representation formulae for solutions described in Section \ref{representation}.  In Section  \ref{lambdaubry} we define
 the $\lambda$--Aubry sets via a condition on cycles. Section \ref{asymptotic} provides in the first part a  summary of the main properties of
 Eikonal equations on networks and then focus on  the behavior of solutions  and $\lambda$--Aubry set as $\lambda \longrightarrow 0$.

 Finally Appendix \ref{immaterial}  collects some basic material on graphs and networks ,  and in Appendix  \ref{basicbis} we provide  some
 proofs of results stated  in Section \ref{boundary}.

\bigskip

\section{Setting of the problem}\label{setting}

We consider  a network $\Gamma$ immersed in
$\R^N$. We denote by $\VV$,  $\mathcal E$ the set of vertices and  arcs, respectively. We also consider the abstract graph $\XX$  underlying
$\Gamma$  with the same vertices of $\Gamma$ and edges that are, loosely speaking, an immaterial copy of the arcs of $\mathcal E$. See
Appendix \ref{immaterial} for more detail and further terminology and notation on graphs and networks.

We are given a
family of  Hamiltonians
\[H_\ga: [0,1] \times \R \to \R \]
indexed by the arcs of the network.  They are unrelated for arcs of different support, and satisfy the compatibility condition

\begin{equation}
H_{- \ga}(s,p) = H_{\ga}(1-s,-p) \qquad\hbox{for any $\ga
\in \EN$}.
  \label{ovgamma}
\end{equation}

We assume the $H_\ga$  to be:

\begin{itemize}
    \item[{\bf (H1)}] continuous in $(s,p)$;
    \item[{\bf (H2)}] coercive in $p$.
\end{itemize}

No convexity conditions are required for the discounted equation. Some additional assumptions  will be introduced for the asymptotic results
of Section \ref{asymptotic} where an Eikonal problem will appear at the limit, as the discount factor goes to $0$. See  hypotheses  {\bf
(H3)},  {\bf (H4)} in Section \ref{asymptotic}

\smallskip

For any given arc $\ga$, we are concerned  with the discounted  equation
\begin{equation}\label{HJg}  \tag{{\bf{HJ}$\ga_\la$}}
   \la \, w  + H_\ga(s,w')=0 \qquad\hbox{in $(0,1)$.}
\end{equation}
The problem we are interested on is a combination of all the \eqref{HJg}. We look for continuous functions $u$ defined on $\Gamma$ such that

\begin{equation}\label{HJ}  \tag{{\bf HJ$\Gamma_\la$}}
  \la \,  u \circ \ga  + H_\ga(s,( u \circ \ga)')= 0 \qquad\hbox{ in $[0,1]$,  for   any $\ga \in \mathcal E$}
\end{equation}
\smallskip
in the viscosity sense, plus suitable conditions at the vertices, as made precise in the forthcoming  Definition \ref{def:HJsol}. We
preliminarily  recall some definition and terminology of viscosity solution theory.
\smallskip
\begin{Definition}
Given a continuous function $w$ in $[0,1]$, we say that a $C^1$
function $\varphi$ is  {\it supertangent} to $w$ at $s \in (0,1)$ if
\[ w = \varphi \;\;\hbox{at $s$} \quad {\rm and}\quad w \leq \varphi \;\;\hbox{in $(s-\de, s+ \de)$ for some $\de >0$.}
\]
The notion of {\it subtangent} is given by just replacing $\leq$  by
$\geq$ in the above formula.

Finally, $\varphi$ is called  {\it constrained subtangent} to $w$ at
$1$ if
\[ w = \varphi \;\;\hbox{at $1$} \quad {\rm and}\quad w \geq \varphi \;\;\hbox{in $(1-\de,1)$ for some $\de >0$.}
\]
A similar notion, with obvious adaptations,  can be given at $t= 0$.\\
\end{Definition}

\smallskip

\begin{Definition} Given  a continuous function $w$ in $[0,1]$, a point $s_0 \in \{0,1\}$,
 we say that it   satisfies
{\it the state constraint boundary condition} for \eqref{HJg} at
$s_0$ if
 \[\la \, \varphi(s_0)+ H_\ga (s_0, \varphi'(s_0)) \geq 0.\]
for any constrained $C^1$ subtangent $\varphi$ to $w$ at $s_0$.
\end{Definition}

\smallskip

\begin{Definition}\label{def:HJsol} We say that $u: \Gamma \longrightarrow \R$ is {\it subsolution} to \eqref{HJ}
if
\begin{itemize}
    \item [{\bf i)}] it is continuous on
    $\Gamma$,
    \item [{\bf ii)}] $s \mapsto u(\ga(s))$ is subsolution to
    \eqref{HJg} in $(0,1)$ for any $\ga \in \EN$.
\end{itemize}
\smallskip
We say that $u$ is {\it supersolution } to \eqref{HJ} if
\begin{itemize}
     \item [{\bf i)}] it is continuous;
     \item [{\bf ii)}] $s \mapsto u(\ga(s))$   is supersolution of
    \eqref{HJg} in $(0,1)$ for any $\ga \in \EN$;
    \item [{\bf iii)}] for
    every vertex $x$ there is at least an arc $\ga$, with $x$ as terminal point, such that   $ u(\ga(s))$
    satisfies the state constraint boundary   condition  for \eqref{HJg}  at $s=1$
\end{itemize}
A function $u$ is said {\it solution} if it is at the same time super and subsolution.
\end{Definition}
Let us observe that {\bf iii)} for supersolutions is actually a partial boundary condition since it is given only at one endpoint. However
when it is combined with a Dirichlet condition at the other endpoint, it gives the uniqueness of the solution as proved in Corollary
\ref{charabis}.  Also notice that in the definition of subsolution no conditions are
required on vertices.

\smallskip

\begin{Remark}\label{neww1}
Passing from   $\ga$ to $-\ga$ and from $H_\ga$ to $H_{-\ga}$, we see that the  condition {\bf iii)} in the definition of supersolution for  a
vertex $x$  can be equivalently given   at $s=0$ considering  the edges with  initial vertex $x$.
\end{Remark}

\smallskip

\begin{Remark}\label{neww2} The condition {\bf iii)} in the  above  definition of supersolution  is the same given in \cite{LionsSouganidis}
at the junction point $0$.  In \cite{LionsSouganidis} the authors   do not impose conditions for the test functions on the other vertices,
but for  the uniqueness principle they  need considering some boundary condition at the other vertices of the junction. We assume condition
{\bf iii)} at any vertex but we get uniqueness of solutions without assuming any additional boundary condition, we do not even single out a
boundary in our network.
\end{Remark}

\bigskip

\section{Local analysis of HJ equations on  arcs}
\label{boundary}

\medskip
We  focus on an arc $\gamma \in \EN$, our treatment is independent of whether  or not
$\ga$ is a cycle.  We  recall some  basic facts about viscosity (sub)solutions to \eqref{HJg},
see for instance \cite{BardiCapuzzo}, \cite{Barles}.

\smallskip

\begin{Theorem}[Comparison Principle] \label{theo:comprinc}
If $u$ is an upper semicontinuous subsolution and $v$ is a lower
semicontinuous supersolution to \eqref{HJg} with $u\le v$ in
$\{0,1\}$, then $u\le v$ in $[0,1]$.
\end{Theorem}

\smallskip

 Given $\al\in \R$, we
define
\begin{eqnarray}
 \um^\ga(s)&=& \sup \{u(s) \mid u \;\hbox{subsolution to \eqref{HJg}}\} \\
  u_\al^\ga (s)&=& \sup \{u(s) \mid u \;\hbox{subsolution to \eqref{HJg}
 with $u(0)\leq \al$} \}\label{defual}
\end{eqnarray}

\smallskip
\begin{Lemma}\label{chara} The  function $\um^\ga$   is characterized by the
property of being a Lipschitz continuous  solution to \eqref{HJg} in
$(0,1)$ satisfying state constraints boundary conditions at $0$ and
$1$.
\end{Lemma}

\smallskip

\begin{Lemma}\label{wellposed}   The function  $u^\ga_\al$ is  a Lipschitz--continuous
 solution to \eqref{HJg} in $(0,1)$ satisfying state constraint boundary conditions at $s=1$.  In addition  $u^\ga_\al$ is equal to $\al$  at
 $s=0$ if and only if $\al \leq \um^\ga(0)$.
\end{Lemma}
The proof of the two above lemmata  is in  Appendix \ref{basicbis}.

\smallskip

\begin{Corollary}\label{corcharabis} The identity $u^\ga_\al \equiv  \um^\ga$ holds true in $[0,1]$ if and only if $\al \geq \um^\ga(0)$.
\end{Corollary}

\smallskip
We deduce from the previous results  the following  characterization of $u^\ga_\al$:

\begin{Corollary}\label{charabis} The function
$u^\ga_\al$, for $\al \leq \um^\ga(0)$,  is the unique solution to \eqref{HJg} satisfying the Dirichlet boundary condition $u^\ga_\al(0)=\al$
and the state constraint
boundary condition at $s=1$.
\end{Corollary}

By slightly adapting the proof of Lemma \ref{chara}, we also have:

\begin{Corollary}\label{charatris}  Let $w$ be a supersolution of \eqref{HJg} with $w(0) = \al$ satisfying  the state constraint
boundary condition at $s=1$, then $w \geq u^\ga_\al$ in $[0,1]$.
\end{Corollary}

\medskip
We introduce the function $u_\al^{-\ga}$ defined as $u^\ga_\al$, but
with the Hamiltonian $H_\ga$ in equation \eqref{HJg} replaced by $H_{- \ga}$.
This function is the analogue of $u^\ga_\al$ on $-\ga$ in the sense that it is the maximal subsolution to $\la \, u + H_{- \ga}(s,u') =0$ taking a
value less than or equal to $\al$ at $0$.

\smallskip

\begin{Remark}\label{remfork} It is apparent that $w(s)$ is subsolution to
\eqref{HJg} with $H_{-\ga}$ in place of $H_\ga$ if and only $s \mapsto
w(1-s)$ has the same property for the original equations. This
 shows that
$u_\al^{-\ga}(1-s)$ is the maximal subsolution to
\eqref{HJg} taking value $\leq \al$ at $s=1$.
In addition, $s \mapsto \um^\ga(1-s)$ is the maximal subsolution to
\eqref{HJg} with $H_{-\ga}$ in place of $H_\ga$.
\end{Remark}

\smallskip

The next result is about  Dirichlet boundary problems. It will  be crucially
used  in the passage from the local problem on the arcs to
the global problem on the network.

\smallskip

\begin{Proposition}\label{fork}
There exists an unique solution $u$ of the equation \eqref{HJg} with
$u(0)=\al$, $u(1)= \be$ if and only if
\begin{equation}\label{fork2}
 \al \leq  u^{-\ga}_\be(1), \;\;  \be  \leq u^\ga_\al(1).
\end{equation}
\end{Proposition}
The proof is in  Appendix \ref{basicbis}.

\smallskip

In the next result we show continuity of $u^\ga_\al(1)$ with respect to
$\al$ plus two  monotonicity properties we will repeatedly exploit
in what follows. We stress in particular  that the strict monotonicity in item {\bf ii)} will play a crucial role in the whole paper.

\begin{Proposition}\label{varalfa}  \hfill
\begin{itemize}
   \item[{\bf i)}]  $\al \mapsto u^\ga_\al(1)$ is  Lipschitz continuous and nondecreasing;
    \item[{\bf ii)}]  $\al \mapsto u^\ga_\al(1)- \al$ is strictly
    decreasing;
    \item[{\bf iii)}]$\lim_{\al \to - \infty} u^\ga_\al(1)-\al = +\infty$ , \;\;\; $ \lim_{\al \to + \infty} u^\ga_\al(1)-\al
    = - \infty $.
\end{itemize}
\end{Proposition}
\begin{proof}

We start from  {\bf ii)}.  We  consider $\be < \al$. The function $s
\mapsto u^\ga_\al + \be -\al$ is a strict subsolution to \eqref{HJg}
taking a value less than or equal to $\be$ at $s=0$. This implies
\begin{equation}\label{varalfa2}
    u^\ga_\al(1)+\be -\al \leq u^\ga_\be(1).
\end{equation}
Arguing as in the proof of Lemma \ref{chara}, we find a
constrained subtangent to  $u^\ga_\al+\be -\al$  at  $s=1$ of the form
\[\varphi(s)= u^\ga_\al(1)+\be -\al + q \,(s-1)\]
for some $q > \max \{p \mid p \in \partial u^\ga_\al(1)\}$  satisfying
\begin{equation}\label{varalfa3}
   \la \, (u^\ga_\al(1) +\be -\al)+ H_\ga(1,q)< 0.
\end{equation}
If equality holds in \eqref{varalfa2} then $\varphi$ is also a
constrained subtangent to $u^\ga_\be$ at $s =1$ and
 \eqref{varalfa3} contradicts $u^\ga_\be$
satisfying state constraint boundary condition at $s =1$, see Corollary \ref{charabis}. Then a
strict inequality must prevail. This shows item {\bf ii)}.

We pass to {\bf i)}. The nondecreasing character of $\al \mapsto
u^\ga_\al(1)$ is a direct consequence of the maximality of $u^\ga_\al$. From this and item { \bf ii)} we derive  for any $\al\ge\be$,
\[0\le u^\ga_\al(1)-u^\ga_\be(1)\le\al-\be,\]
which implies the claimed continuity.

  To prove {\bf iii)},
 we recall  that  by Corollary \ref{corcharabis}
\[u^\ga_\al(1) - \al= \um^\ga(1) - \al  \qquad\hbox{for $\al$ sufficently large
,}\] which gives  the claimed negative divergence as $\al \to +
\infty$. Given any $p_0 > 0$, we consider $\al$ with \[ - \max_{s \in
\R} \{  H_\ga(s,p_0)\}  \geq \la \, ( \al + p_0),\] then $ s \mapsto \al
+s \,p_0$ is subsolution to \eqref{HJg}, and consequently
$u^\ga_\al(1)- \al \geq p_0$. This implies the claimed positive
divergence as $\al \to - \infty$.
\end{proof}

\smallskip
We  derive:

\begin{Corollary}\label{corvaralfa}
There exists one and only one $\al$ such that $u^\ga_\al(1)=\al $, and it  satisfies $\al \geq - \frac 1\la \, \max\limits_s H_\ga(s,0)$.
\end{Corollary}
\begin{proof} If $\al = - \frac 1\la \, \max\limits_s H_\ga(s,0)$ then the function constantly equal to $\al$ is subsolution  to \eqref{HJg}.
Consequently  $u^\ga_\al(1) \geq \al$, and  the conclusion follows from Proposition \ref{varalfa} {\bf ii)}, {\bf iii)}.
\end{proof}

\smallskip
\begin{Remark}\label{loop1}  According to Proposition \ref{fork},  the equation \eqref{HJg} admits  a periodic solution in $(0,1)$, namely
attaining the same value at $0$ and $1$, if and only if the boundary value is less than or equal to the $\al$ appearing in the statement of
Corollary \ref{corvaralfa}.

\end{Remark}

\medskip

We introduce the Eikonal equation
\begin{equation}\label{HJeikloc} \tag{{\bf HJ$_\ga$}}
    H_\ga(s,u') = 0 \quad\hbox{$s \in (0,1)$}
\end{equation}
 under the additional assumptions {\bf (H3)}, {\bf (H4)},  see Section   \ref{asymptotic}  for a precise statement of these conditions and a
 quick review  of Eikonal equation on networks. We   define
\[a_\ga= \max_s \, \min_p H_\ga(s,p)\]

\begin{Lemma}\label{previous}  If $0 \geq a_\ga$,  then   there is  a function  $v$   such that $\al + v$  is the maximal subsolution to
\eqref{HJeikloc} taking the value $\al$ at $0$, for any $\al \in \R$. It is in addition  a Lipschitz continuous solution of \eqref{HJeikloc}.
\end{Lemma}
\begin{proof}
 See Proposition 5.6 in \cite{SiconolfiSorrentino}. The solution $v + \al$ is given by formula (20) in \cite{SiconolfiSorrentino} with $\al$
 in place of $w(0)$ and $0$ in place of $a$.

\end{proof}

\smallskip

We study the asymptotic behavior  of solutions to \eqref{HJg} as $\la \to 0$.  Given a  positive  infinitesimal sequence $\la_n$,  we indicate
by $\um^{\la_n}$, $u_\al^{\la_n}$ the maximal solution to \eqref{HJg}, with $\la_n$ in place of $\la$,  and the maximal solution among those
taking the value $\al$ at $s=0$, respectively.  The function $v$ is defined as in the statement of Lemma \ref{previous}.  The proof of the
following  result is in Appendix \ref{basicbis}.

\begin{Lemma}\label{asy}  Let  $\al_n$ be a sequence converging to  some $\al \in \R$. If    $\um^{\la_n}(0)\ge\al_n$ for $n$ sufficiently
large, then   $u_n = u_{\al_n}^{\la_n}$  uniformly converges in  $[0,1]$      to $\al + v$.
\end{Lemma}

\bigskip

\section{Discrete functional equations}\label{statgraph}

\medskip

We introduce  a discrete functional  equation  on $\VV$  suitably related to \eqref{HJ}.   The
relation   is made clear  in Theorem \ref{theo:ponte}, Proposition \ref{pontebis}.
\smallskip

For  $e = \Psi^{-1}(\ga)$, we set

\begin{eqnarray*}
  \rho(\al,e) &=& u_\al^\ga(1) \\
 \un\al(e) &=& \um^\ga(0) \\
  \ov\al(e) &=& \um^\ga(1),
\end{eqnarray*}

\smallskip

We record for later use:

\begin{Proposition}\label{fondadis} For any $e \in \EE$ we have
 \[\ov \al(e) = \un \al(-e)= \rho(\un
 \al(e),e)\]
\end{Proposition}
 \begin{proof} The equalities  in the statement directly   come from the definitions of $\un\al$, $\ov \al$, $\rho$
and Remark \ref{remfork}.
\end{proof}

\smallskip

The  discrete functional equation in $\VV$ is defined as follows:

\smallskip
\begin{equation}\label{HJdis} \tag{{\bf DFE$_\la$}}
  U(x) = \min_{e \in -\EE_x} \rho( U(\oo(e)),e).
\end{equation}

We say that $U:\VV \longrightarrow \R$ is a {\it subsolution} (resp. {\it supersolution}) to
\eqref{HJdis} if
\begin{equation}\label{subsolbis}
 U(x) \leq \;\hbox{(resp. $\geq$ )} \;  \min_{e \in -\EE_x} \rho( U(\oo(e)),e). \qquad\hbox{for
any $x \in \VV$.}
\end{equation}
A solution is at the same time sub and supersolution. See \eqref{neww3} in the Appendix for the definition of $-\EE_x$. Notice that in
accordance with condition {\bf iii)} in Definition \ref{def:HJsol}  of the supersolution on the network, we have considered in \eqref{HJdis}
only the edge  ending at $x$. As pointed out in Remark \ref{neww2}, it is equivalent to instead consider arcs (in Definition \ref{def:HJsol})
and edges (in \eqref{HJdis}) starting at $x$

\medskip

The following results provide the bridge  linking  \eqref{HJdis} to
\eqref{HJ}.

\smallskip
\begin{Theorem}\label{theo:ponte}
A solution $U$ to \eqref{HJdis} can be uniquely extended to a
solution $u$ to \eqref{HJ}. Conversely, given a solution $u$ to
\eqref{HJ},  $U=u|_{\VV}$ is a solution to \eqref{HJdis}.
\end{Theorem}
\begin{proof}
Assume that $U$ solves \eqref{HJdis}. Let  $e$ be an edge in $\EE$.
We set, to ease notations, $\ga=\Psi(e)$,  $\al = U(\oo(e))$, $\be =
U(\tt(e))$. By the very definition of subsolution to \eqref{HJdis}
and $\rho$ we have
\begin{gather*}
\be \le \rho(\al, e)=u^\ga_{\al}(1)\\
\al\le \rho(\be,e)= u^{-\ga}_\be(1)
\end{gather*}
and this implies, thanks to Proposition \ref{fork}, that there is a
unique solution $w$ to \eqref{HJg} with $w(0)=\al$ and $w(1)=\be$.
We have in addition that for any $x\in V$ there exists $e_0 \in
-\EE_x$ with
\[ U(x)=\rho(U(\oo(e_0)),e_0).\]
This implies that $U$  can be uniquely extended as the maximal
subsolution to  \eqref{HJg}  less than or equal to $U(\oo(e_0))$ at
$s=0$.  It is by Corollary  \ref{charabis}  a solution to
\eqref{HJg} and  satisfies the state constraint boundary condition
at $s =1$.  This shows the first part of the assertion. Conversely,
assume that $u$ is a solution to \eqref{HJ}, and set $U=u|_\VV$. We
deduce from the definition of $\rho$ and Proposition \ref{fork}
\begin{equation}\label{ponte1}
  U(x)\le\min_{e\in- \EE_x} \rho(U(\oo(e)),e) \qquad\text{ for any
}x\in\VV.
\end{equation}
Taking into account that  $u$ satisfies condition {\bf iii)} in the
definition of solution to \eqref{HJ},  we find in force of
Corollary \ref{charabis}  for any $x\in\VV$ an $e_0 \in - \EE_x$
for which formula \eqref{ponte1} holds with equality.
 This shows that $U$ solves \eqref{HJdis} and concludes the
proof.
\end{proof}

As a consequence of the very definition of $\rho$ and Corollary \ref{charatris}, we also have

\begin{Proposition}\label{pontebis} The trace on $\VV$ of any subsolution (resp. supersolution) to   \eqref{HJ} is a subsolution (resp.
supersolution ) of \eqref{HJdis}

\end{Proposition}

\smallskip

 We  establish a  comparison principle for \eqref{HJdis}.
\smallskip

\begin{Theorem}\label{sol!}  Let $U$, $W$ be a subsolution and a supersolution, respectively, to
\eqref{HJdis}. Then $U \leq W$.
\end{Theorem}
\begin{proof} Assume by contradiction that
$\max_\VV U - W >0$, and  denote by  $x_0$  a  corresponding maximizer.
In force of the very definition of subsolution and supersolution,  there
is $e_0 \in - \EE_{x_0}$  with
\begin{eqnarray*}
  U(x_0) &\leq&   \rho(U(\oo(e_0)), e_0)\\
  W(x_0) &\geq&  \rho(W(\oo(e_0)), e_0).
\end{eqnarray*}
By subtracting the above relations,  we obtain
\begin{equation}\label{sol!1}
 U(x_0)- W(x_0) \leq   \rho(U(\oo(e_0)), e_0)  -
\rho(W(\oo(e_0)), e_0) ,
\end{equation}
 and so,  bearing in mind that $U(x_0) > W(x_0)$, we get
\begin{equation}\label{sol!2}
\rho(U(\oo(e_0)), e_0)  - \rho(W(\oo(e_0)), e_0) >0.
\end{equation}
Since  $ \rho(\cdot,e_0)$  is  nondecreasing by Proposition \ref{varalfa}, we
derive from \eqref{sol!2}  $U(\oo(e_0))> W(\oo(e_0))$.  Thus,
exploiting the strictly decreasing character of $ \al \mapsto \rho(\al,
e_0) - \al$, we further get from \eqref{sol!1}
\[U(\oo(e_0))- W(\oo(e_0)) >  \rho(U(\oo(e_0)), e_0)  -
\rho(W(\oo(e_0)), e_0) \geq  U(x_0) - W(x_0)  \]
which contradicts $x_0$ being a maximizer  of $U-W$ in $\VV$.
\end{proof}

\smallskip

We  derive as a consequence:

\begin{Theorem}\label{theo:discunisol} The discounted discrete equation  can have at most one
solution.
\end{Theorem}

\smallskip

By combining Theorem \ref{sol!} and Proposition \ref{pontebis}, we finally  state a comparison principle for \eqref{HJ}.

\begin{Theorem}\label{compara}Let $u$, $w$ be sub and supersolution of \eqref{HJ}, the $u \leq w$ in $\Gamma$.
\end{Theorem}
\begin{proof} By Proposition \ref{pontebis} the traces of $u$, $w$  on $\VV$ are sub and supersolution to  \eqref{HJdis}, respectively. By
Theorem \ref{sol!}
$u|_{\VV} \leq w|_{\VV}$.  This gives the assertion in force of Theorem \ref{theo:comprinc}.

\end{proof}

\section{Analysis of the discrete equation}\label{discrete}

In this section we extend the definition of
$\rho$ from edges to general paths  via an inductive procedure on   the length of  paths. We furthermore define  some related quantities.
\smallskip

\begin{Definition} \label{functional}
Given $\al \in \R$ and a path $\xi$, we define
\[\rho(\al,\xi)= \rho(\al,e) \qquad\hbox{ if $\xi=e$ .}\]
If $\xi=(e_i)_{i=1}^M$, for $M >1$, we set
$\bar\xi=(e_i)_{i=1}^{M-1}$ and define
\[\rho(\al,\xi)=  \rho(\rho(\al,\bar\xi),e_M).\]
\end{Definition}

\smallskip
The following concatenation formula is
inherent to the definition. Let $\xi$, $\eta$  be paths with
$\tt(\xi)= \oo(\eta)$ then
\begin{equation}\label{concatena}
   \rho(\al, \xi \cup \eta)= \rho(\rho(\al,\xi),\eta) \qquad\hbox{for any $\al$.}
\end{equation}

\smallskip

Taking into account  that the property of being continuous is stable
for composition of functions, we get from Proposition
\ref{varalfa}:

\smallskip

\begin{Proposition}\label{preprecycledis} Given any path $\xi$, the function
\[\al  \mapsto \rho(\al,\xi)\]
is continuous.
\end{Proposition}

\smallskip

The next Proposition is a direct consequence of Proposition \ref{varalfa} and will be repeatedly used in what follows.
\begin{Proposition}\label{precycledis} The following monotonicity  properties hold for any
path $\xi$
\begin{itemize}
    \item[{\bf i)}] $\al \mapsto  \rho(\al,\xi)$ is nondecreasing;
    \item[{\bf ii)}]  $\al \mapsto \rho(\al,\xi) - \al$ is strictly
    decreasing.
\end{itemize}
\end{Proposition}
\begin{proof} We prove both items arguing by induction on the length
of the path. If it is $1$, and so the path reduces to an edge, the
statement  is a direct consequence of the definition of $\rho$
and Proposition \ref{varalfa}. We assume the assertion to be true for
any path with length less than $M$ and  show it for  $\xi :=(e_i)_{i=1}^M$. By the very definition of $\rho$
\begin{equation}\label{precycledis0}
 \rho(\al,\xi) =   \rho(\rho(\al, \bar\xi),e_M),
\end{equation}
where $\bar\xi= (e_i)_{i=1}^{M-1}$.
 The functions $ \al \mapsto  \rho(\al,\bar\xi)$
and $\al \mapsto  \rho(\al, e_M)$ are nondecreasing by the
inductive step,  and $\rho(\cdot,\xi)$ is therefore nondecreasing as composition of
nondecreasing functions. This concludes the proof of item {\bf i)}.
To show {\bf ii)}, we argue again by induction. Given $\be < \al$, we have by item {\bf i)} $\rho(\be,\ov\xi) \leq \rho(\al,\ov\xi)$,
exploiting this inequality,  and the  inductive step, we get
\begin{eqnarray*}
 \rho(\al, \ov\xi) - \rho(\be,\ov\xi) &<&   \al - \be\\
   \rho(\rho(\al, \bar\xi),e_M) - \rho(\rho(\be, \bar\xi),e_M) &\leq& \rho(\al, \ov\xi) - \rho(\be, \ov\xi).
\end{eqnarray*}
By combining the above inequalities, we obtain
\[ \rho(\al,\xi) - \rho(\be,\xi) < \al -\be\]
which gives {\bf ii)}.
\end{proof}

\smallskip

The next result is a generalization to paths  of Corollary \ref{corvaralfa}.   It   has a crucial relevance since the fixed points  of $\rho$
will play a key role in our analysis.

\begin{Corollary}\label{cycledis} For any path  $\xi$ there exists one and only one $\al \in \R$
with $\rho(\al,\xi) = \al$.
\end{Corollary}
\begin{proof} We  have by the definition of $\ov\al$ and $\rho$
\begin{equation}\label{cycledis0}
\rho(\al,e) \leq  \ov\al(e) \qquad\hbox{for any $e \in\EE$, $\al \in
\R$.}
\end{equation}
Let $\xi=(e_i)_{i=1}^M$ and  $\bar\xi=(e_i)_{i=1}^{M-1}$.   We  get in force of \eqref{cycledis0} and
the concatenation formula  \eqref{concatena}
 \begin{equation}\label{cycledis1}
\rho(\al,\xi)=  \rho(\rho(\al,\bar\xi),e_M) \leq \ov\al(e_M) \qquad\hbox{ for any $\al \in \R$.}
\end{equation}
Taking into account Corollary \ref{corvaralfa},  we set
\[\al_0 = \min\{\al \mid \rho(\al,e_i)=\al ,\; i=1,\cdots,M\}.\]
We claim that
\begin{equation}\label{cycledis2}
   \rho(\al,\xi) > \al \qquad\hbox{for $\al < \al_0$.}
\end{equation}
We  fix $\al  > \al_0$ and prove the claim  arguing by induction on the length of the curve. If the
length is $1$, say $\xi=e$, then \eqref{cycledis2} holds because of
the strict monotonicity of $ \al \mapsto \rho(\al,e) - \al$.  Assuming the property
true for curves of length less than $M$, we get
$\rho(\al,\bar\xi) >\al$ and consequently by the nondecreasing character of $\rho(\cdot,e_M)$ and \eqref{concatena}
\[\rho(\al,\xi) =\rho(\rho(\al,\bar\xi),e_M) \geq  \rho(\al,e_M) > \al,\] proving the claim. Relations
\eqref{cycledis1}, \eqref{cycledis2} plus continuity and
monotonicity of $\rho(\cdot,\xi)$, see Propositions
\ref{precycledis}, \ref{preprecycledis}, give the assertion.
\end{proof}
In what follows, we will exploit the property highlighted by the above proposition solely for cycles.

\smallskip

\begin{Definition}\label{defbe}
 Given a cycle  $\xi$, we  define  $\be(\xi)$ to be  the unique fixed point of
\[\al \mapsto \rho(\al,\xi).\]
\end{Definition}

\smallskip

\begin{Proposition}\label{postcycledis}  For any edge $e$, the cycle $\xi=(e,-e)$
satisfies
\[\be(\xi) = \un\al(e).\]
\end{Proposition}
\begin{proof}  We have
\[ \rho(\un\al(e),\xi) =  \rho(\rho(\un \al(e),e), -e)\]
and we derive, taking into account Lemma \ref{fondadis}
\[ \rho(\un\al(e),\xi) =  \rho(\un\al(-e), -e)=\un\al(e).\]
\end{proof}

\smallskip

\begin{Remark}
It is worth pointing out that $\be(\xi)$, see Definition \ref{defbe},  also depends on the initial point of the cycle. In other terms, if we
consider another cycle $\eta$ with the same edges as $\xi$ but different initial point then in general $\be(\xi) \neq \be(\eta)$. For example,
if we define, for a given edge $e$,  $\xi=\{e,-e\}$ and $\eta= \{- e,e\}$ then, according to Proposition \ref{postcycledis}, $\be(\xi)=
\un\al(e)$ and $\be(\eta)=\ov\al(e)$, which are clearly in general different. In what follows when we will say that a cycle is a based on a
certain vertex, we will mean that the vertex is the initial point of the cycle.
\end{Remark}

\bigskip

\section{Existence of solutions of \eqref{HJdis}, \eqref{HJ} and representation formulae}\label{representation}
We   show that a solution  to  \eqref{HJdis}  does exist
providing  a representation formula.   We define a function $f:
\VV \to\R$ via
\[f(x) = \inf \{\be(\xi)\mid\hbox{for some cycle
$\xi$ based on $x$}\}.\] The definition is well posed thanks to
Corollary \ref{cycledis}.  We set for $x \in \VV$
\begin{equation}\label{soldis}
   U(x) = \inf \{\rho(f(\oo(\xi)),\xi) \mid \,\xi
   \;\hbox{path with} \; \tt(\xi)=x\}.
\end{equation}

Since for any vertex $x$, any cycle based on $x$ is an admissible path for \eqref{soldis}, it is clear that
\[U(x) \leq  f(x).\]

We have

\begin{Theorem}\label{esiste} The function $U$ defined in \eqref{soldis} is
solution to \eqref{HJdis}.
\end{Theorem}
\smallskip

The rest of the section is devoted to the deduction of some
properties of $f$ and $U$, and to the proof of Theorem
\ref{esiste}.

\smallskip

\begin{Proposition}\label{presiste} We have
\begin{equation}\label{corva0}
  - \frac 1\la \, \max\limits_{e,s} H_{\Psi(e)}(s,0) \leq  f(x) \leq \min_{e \in \EE_x}\un\al(e)   \qquad\hbox{for any $x \in \VV$.}
\end{equation}
\end{Proposition}
\begin{proof} The rightmost inequality of the formula in the statement  is a direct consequence of the
definition of $f$ and Proposition \ref{postcycledis}.
We set $\ov\al = - \frac 1\la \, \max\limits_{e,s} H_{\Psi(e)}(s,0)$, and claim that
\begin{equation}\label{corva1}
 \rho(\ov\al,\xi) \geq \ov\al \qquad\hbox{for any path $\xi$.}
\end{equation}
Were the claim true, we derive from it, because of the strict monotonicity of $\al \mapsto \rho(\al,\xi) - \al$, $\be(\xi) \geq \ov\al$ for
any cycle $\xi$. This in turn implies the leftmost inequality in \eqref{corva0}. We prove \eqref{corva1} arguing inductively on the length of
paths.  It is true if the length is $1$ in force of Corollary`\ref{corvaralfa}. We take a general path $\xi=(e_i)_{i=1}^M$ and set
$\ov\xi=(e_i)_{i=1}^{M-1}$.  By inductive step $\rho(\ov\al,\bar\xi) \geq \ov\al$  and $\rho(\ov\al,e_M) \geq \ov\al$. Exploiting the
monotonicity of  $\rho(\cdot,e_M)$, we have
\[ \rho(\ov\al,\xi)= \rho(\rho(\ov\al,\bar\xi),e_M) \geq \rho(\ov\al,e_M) \geq \ov\al.\]
This  concludes the proof.
\end{proof}

\smallskip

\begin{Proposition}\label{cyclic}   The infimum in the definition of $U$ is realized by  a simple path  with terminal vertex $x$, for any $x
\in \VV$.
\end{Proposition}
\begin{proof}
 We fix $x$ and a  path   $\xi$ with terminal vertex    $x$, and set, to ease notation, $\ov\al=f(\oo(\xi))$. Let us assume that there is a
 cycle $\eta$
properly contained in $\xi$ with
\begin{equation}\label{cycclic1}
\oo(\eta) \neq \oo(\xi)  \qquad\hbox{and} \qquad \tt(\eta) \neq \tt(\xi).
\end{equation}
The path $\xi$ can be consequently written in the form
\[\xi=\xi_1 \cup \eta \cup \xi_2\]
where $\xi_1$, $\xi_2$, $\eta$  satisfy $\tt(\xi_1)=\oo(\xi_2)= \oo(\eta)$. We have by the concatenation formula
\begin{equation}\label{cycclic2}
 \rho(\ov\al,\xi)= \rho(\rho(\rho(\ov\al),\xi_1),\eta),\xi_2).
\end{equation}
If  $\rho(\ov\al,\xi_1) \geq \be(\eta)$
then  by the usual monotonicity property
\[\rho(\rho(\ov\al,\xi_1),\eta) \geq \rho(\be(\eta),\eta) =\be(\eta)\]
which implies, taking also into account   \eqref{cycclic2}  and the definition of $\ov\al$
\begin{equation}\label{cycclic3}
\rho(f(\oo(\xi)),\xi) \geq \rho(\be(\eta),\xi_2) \geq \rho(f(\oo(\xi_2)),\xi_2)).
\end{equation}
If instead   $\rho(\ov\al,\xi_1) < \be(\eta)$, then by the strict monotonicity of $\al \mapsto \rho( \al,\eta) -\al$, we have
\[\rho(\rho(\ov\al,\xi_1),\eta) > \rho(\ov\al,\xi_1)\]
and by  \eqref{cycclic2} and the definition of $\ov\al$, we further get
\begin{equation}\label{cycclic4}
 \rho(f(\oo(\xi)),\xi) > \rho(\rho(\ov\al,\xi_1),\xi_2)= \rho(\ov\al, \xi_1 \cup \xi_2)= \rho(f(\oo(\xi_1\cup\xi_2), \xi_1 \cup \xi_2).
\end{equation}
Taking into account  \eqref{cycclic3}   \eqref{cycclic4},  we realize that the cycle $\eta$ can be removed without affecting  the infimum in
the definition of $U(x)$.  By slightly adapting the argument, we reach the same conclusion getting rid of condition \eqref{cycclic1}. The
procedure can be repeated for all other cycle properly contained in $\xi$.  We therefore  see that
\[U(x)= \min \{  \rho(f(\oo(\zeta)), \zeta) \mid \zeta
\;\hbox{simple path with $\tt(\zeta)=x$}\},\]
where the minimum in the above formula is justified by the fact that the simple paths are finite.
This ends the proof.

\end{proof}

\smallskip
 By following the same argument as in Proposition \ref{cyclic} we can also show

\begin{Corollary} \label{corcyclic} Assume that for a given $x$
\[U(x)= \rho(f(\oo(\xi)),\xi) \qquad\hbox{for some path $\xi$ with $\tt(\xi)=x$.}\]
Then there exists a simple path $\zeta$ with $\oo(\zeta)= \oo(\xi)$, $\tt(\zeta)=y$ such that
\[U(x)= \rho(f(\oo(\zeta),\zeta).\]
\end{Corollary}

\smallskip

\begin{proof} ({\bf of Theorem \ref{esiste}}) \;
  We fix $x \in \VV$ and  $e \in - \EE_x$.  By Proposition \ref{cyclic}, there is a simple path $\xi$ ending at $\oo(e)$ with
\[  U(\oo(e)) =   \rho(f(\oo(\xi)),\xi).\]
By the very definition of $U$ and  the concatenation principle \eqref{concatena}, we have
\[ U(x) \leq    \rho(f(\oo(\xi)),\xi \cup e)= \rho(U(\oo(e),e).\]
 This shows that $U$ is subsolution.
Taking  again into account Proposition  \ref{cyclic},  we proceed  denoting by $\eta= (e_i)_{i=1}^M$  a simple path with terminal
point $x$ satisfying
\[ U(x) =  \rho(f(\oo(\eta),\eta).\]
  We set
$\bar\eta= (e_i)_{i=1}^{M-1}$, and derive from   concatenation
formula, monotonicity and definition of $U$
\[ U(x) = \rho(\rho(f(\oo(\eta), \bar\eta),e_M) \geq \rho(U(\oo(e_M)),e_M)\]
Knowing that $U$ is subsolution and $e_M \in - \EE_x$,
equality must prevail in the above formula, showing that $U$ is
actually a solution, as was claimed.

\end{proof}

\smallskip
\begin{Remark}\label{loop2} If $e$ is a loop with vertex $x$ then clearly $U(x) \leq \be(e)= \be(-e)$, see Remark \ref{loop1}, if there is a
strict inequality then the edge $e$ (resp.the arc $\Psi(e)$) can be removed from the graph (resp. from the network) without affecting  the
solution of \eqref{HJdis} (resp.  the solution of \eqref{HJ} on the  arcs different from $\Psi(e)$). A similar phenomenon takes place for the
Eikonal equation on graphs/networks, see Remark 6.17 in \cite{SiconolfiSorrentino}.

\end{Remark}

\smallskip

Combining the previous result with Theorems \ref{theo:ponte} and \ref{theo:discunisol} we get
\begin{Theorem}
There is one and only one solution to \eqref{HJ}, and its restriction  to  $\VV$  coincide with the function
$U$ defined in \eqref{soldis}.
\end{Theorem}

\bigskip

\section{$\la$--Aubry sets}\label{lambdaubry}

We define in this section  the $\la$--Aubry sets, an analogue to the Aubry sets introduced   for the Eikonal problem, see Section
\ref{asymptotic}. These  sets allow writing a new representation formula for solutions to \eqref{HJdis}, and will play a  role in the
asymptotic problem we will deal with in the next section.

\begin{Definition}

The \emph{(projected) $\la$--Aubry set} is given by
\[\A_\la= \{ y \in \VV \mid U(y)= \be(\xi) \; \hbox{for some cycle $\xi$ based on $y$.}\}
\]
\end{Definition}

\smallskip

\begin{Proposition}\label{spring} Given $y \in \A_\la$, then any  cycle $\xi=(e_i)_{i=1}^M$  based on $y$ with $U(y)= \be(\xi)$ satisfies
\begin{eqnarray}
  U(\oo(e_j)) &=& \be \big ( (e_i)_{i=j}^M \cup (e_i)_{i=1}^{j-1} \big )  \label{spring1} \\
  U(\oo(e_j)) &=&  \rho \left (U(\oo(e_k)), (e_i)_{i=k}^{ j-1} \right ) \label{spring11}
\end{eqnarray}
for any $j, \, k =1, \cdots, M$, $k \leq j$.
\end{Proposition}
\begin{proof} We start proving \eqref{spring1}.    Taking into account that $U$ is solution to \eqref{HJdis}, we have
\begin{equation}\label{spring2}
 U(y) \leq \rho(U(\oo(e_M)),e_M).
\end{equation}
We set $\eta= (e_i)_{i=1}^{M-1}$, $\zeta= e_M \cup \eta$, it is clear that $\zeta$ is a cycle based on $\oo(e_M)$. By the concatenation
formula
\begin{equation}\label{spring20}
 U(y)= \rho(U(y), \xi)= \rho(\rho(U(y), \eta),e_M).
\end{equation}
We then derive from \eqref{spring2}, \eqref{spring20}  and the monotonicity of $\rho(\cdot, e_M)$
\[\rho(U(y),\eta) \leq U(\oo(e_M))\]
which in turn implies, due to  $ U(y)= \be(\xi) \geq f(y)$,
\[\rho(f(y),\eta) \leq \rho(U(y), \eta) \leq  U(\oo(e_M))\]
We then have  by the very definition of $U$, and  \eqref{spring20}
\begin{equation}\label{spring100}
 U(\oo(e_M)) =  \rho(U(y), \eta) \qquad\hbox{and} \qquad U(y) = \rho(U(\oo(e_M)),e_M).
\end{equation}
We finally get
\[\rho(U(\oo(e_M)), \zeta)= \rho(\rho(U(\oo(e_M)), e_M), \eta)= \rho(U(y),\eta)= U(\oo(e_M))\]
and consequently
\[ U(\oo(e_M) = \be(\zeta)\]
or, in other term, formula \eqref{spring1} with $j=M$.  It can be extended to all $j$ by iterating backward the above argument.

We proceed proving \eqref{spring11}. We set
\[\al =  \rho \left (U(\oo(e_k)), (e_i)_{i=k}^{ j-1} \right ).\]
By \eqref{spring1} we have
\[ \rho \left (\al, (e_i)_{i=j}^M \cup (e_i)_{i=1}^{k-1} \right ) = U(\oo(e_k))\]
and accordingly by the concatenation formula
\begin{eqnarray*}
  \rho \left ( \al, (e_i)_{i=j}^M \cup (e_i)_{i=1}^{j-1} \right )  &=&  \rho \left (  \rho \left (\al, (e_i)_{i=j}^M \cup (e_i)_{i=1}^{k-1}
  \right ), (e_i)_{i=k}^{ j-1} \right ) \\
   &=& \rho \left (U(\oo(e_k)), (e_i)_{i=k}^{ j-1} \right )= \al.
\end{eqnarray*}
This implies  by \eqref{spring1} that $\al = U(\oo(e_j))$, as was claimed.

\end{proof}

The above assertion can be slightly strengthen.

\begin{Corollary}\label{corspring} Given $y \in \A_\la$, there exists a circuit $\zeta$  based on $y$ with $U(y)= \beta(\zeta)$. It therefore
enjoys  the same properties stated for $\xi$  in Proposition \ref{spring}.
\end{Corollary}
\begin{proof} We adopt the same notation of Proposition \ref{cyclic}. We denote by $\xi$ a cycle based on $y$ with $U(y)= \be(\xi)$. We
assume that there is a cycle $\eta$  properly contained in $\xi$ satisfying  condition \eqref{cycclic1}. Since  $\oo(\xi_2)=\oo(\eta)$ we get
thanks to the concatenation principle and \eqref{spring11}
\[U(y)=\rho(U(\oo(\xi_2)),\xi_2)=\rho(U(\oo(\eta)),\xi_2)=\rho(U(y),\xi_1\cup\xi_2).\]
This shows that $U(y)=\be(\xi_1\cup\xi_2)$. By slightly adapting the argument, we reach the same conclusion getting rid of condition
\eqref{cycclic1}. This procedure can be repeated for all other cycles properly contained in $\xi$, and we end up with a circuit $\zeta$
satisfying the assertion.
\end{proof}

\smallskip

The next Proposition provide a further representation formula  for the solution of \eqref{HJdis} and shows  that the $\la$--Aubry sets are
nonempty. The argument is reminiscent of that of Proposition 6.15 in \cite{SiconolfiSorrentino}.

\begin{Proposition}\label{prop:lambdaAubry}
The $\la$--Aubry set is nonempty. Moreover, if $U$ is the solution of \eqref{HJdis} then the following formula holds true
\begin{equation}\label{lambdaAubry}
  U(x) = \min \{ \rho(U(y), \zeta) \mid y \in \A_\la, \, \zeta\text{ simple path that links $y$ to $x$}\}.
\end{equation}
If $y$, $\zeta= (e_i)_{i=1}^M$ realize the minimum in \eqref{lambdaAubry},we in addition have
\begin{equation}\label{lambdaAubrybis}
U(\oo(e_j))= \rho \left ( U(y), (e_i)_{i=1}^{j-1} \right ) \qquad\hbox{for any $j= 2, \cdots, M$.}
\end{equation}
\end{Proposition}
\begin{proof}
Since $U$ is solution, then there exists for any $x \in \VV$ an edge $e \in - \EE_x$ with
\[U(x)=\rho(U(\oo(e)),e). \]
By iterating backward the previous procedure and using the concatenation formula,  we can construct  a path $\xi$ of any possible length, with
$\tt(\xi)=x$  such that
\[U(x)=\rho(U(\oo(\xi)),\xi).\]
Since the set $\EE$ is finite,  we will find, by going on in the iteration, a cycle $\eta$ contained in $\xi$ such that, by construction
\[U(\oo(\eta))= U(\tt(\eta))=\rho(U(\oo(\eta)),\eta)\]
which implies $U(\oo(\eta))= \be(\eta)$ and consequently  that $y:= \oo(\eta) \in \A_\la$. We denote by $\zeta$ the portion of $\xi$ after
$\eta$. It is a simple path, up to suitable choice of the cycle $\eta$,  joins $y$ to $x$, and in addition
\[U(x)= \rho(U(y),\zeta).\]
This relation   shows \eqref{lambdaAubry}. Formula \eqref{lambdaAubrybis} is a direct consequence of the construction of $\zeta$.
\end{proof}

\smallskip
\begin{Remark} If $e$ is a loop with vertex $x$ and $U(x)= \be(e)= \be(-e)$ then apparently $x \in \A_\la$.  By combining it with Remark
\ref{loop2}, we can say that if on the contrary $x \not\in \A_\la$ then any loop based on $x$ can be removed from the graph without affecting
the solution $U$.

\end{Remark}

\bigskip

\section{Asymptotic as $\la \longrightarrow 0$} \label{asymptotic}

In this section we will study the asymptotic behavior of the
solutions to  \eqref{HJ}, \eqref{HJdis}  and the corresponding $\la$--Aubry sets as $\la$ tends to $0$, assuming that the
Hamiltonians $H_\ga$ satisfy, in addition to {\bf (H1)} and {\bf
(H2)}, the conditions  {\bf (H3)} and {\bf (H4)}, see Subsection \ref{subsec:eik}.  We  plan to perform in a subsequent paper a more complete
analysis of the issue with the aim of recovering in our setting the uniqueness of the limit  established in \cite{DFIZ}.

\medskip

\subsection{Eikonal equations on networks}\label{subsec:eik}

We summarize in this subsection some material taken from \cite{SiconolfiSorrentino}  needed  for the forthcoming convergence results. We
consider   the Eikonal  problem  on $\Gamma$ assuming, beside  {\bf (H1)}, {\bf (H2)}, the following additional conditions

\begin{itemize}
    \item[{\bf (H3)}] for any $x \in \Gamma$, $\ga \in \mathcal E$,   $H_\ga(x,\cdot)$ is quasiconvex with
    \[\mathrm{int} \{p \mid H_\ga(x,p) \leq a\}=  \{p \mid H_\ga(x,p) < a\} \quad\hbox{for any $a \in \R$,}\]
    where int stands for the interior.
    \item[{\bf (H4)}]  given any $\ga \in \EN$,  the map  $s \mapsto  \min_{p \in \R} H_\ga(s,p)$
    is constant  in $[0,1]$.\\
\end{itemize}

\smallskip
\begin{Remark} Assumption  {\bf (H4)} can be actually formulated in a slightly weaker way, see \cite{SiconolfiSorrentino}, We have chosen the
above version for simplicity.

\end{Remark}

\smallskip

We consider for any given arc $\ga$  the family of  Eikonal equations
\[ H_\ga(s,w')=a \qquad\hbox{in $(0,1)$,}\]
with $ a \in \R$.  We look for continuous functions $v$ defined on $\Gamma$ such that
\[H_\ga(s,( v \circ \ga)')= a \qquad\hbox{ in $[0,1]$,  for   any $\ga \in \mathcal E$}\]
\smallskip

The definition of (sub/super) solution is given as in  Definition \ref{def:HJsol} with obvious adaptations.

\begin{Proposition} There exists one and only one value of $a$, called {\em critical}, such that the above equation on $\Gamma$ admits
solutions.
\end{Proposition}

We assume throughout the paper, without any loss of generality, that the critical value is $0$.   It is then clear that
\[  0  \geq   \max_{\ga \in \EN} \, \min_{p \in \R} H_\ga(0,p).\]

We focus on the critical equations

\begin{equation}\label{HJeikg} \tag{{\bf{HJ}$\ga$}}
    H_\ga(s,w')=0 \qquad\hbox{in $(0,1)$,}
\end{equation}
and
\begin{equation}\label{HJeik}  \tag{{\bf HJ$\Gamma$}}
 H_\ga(s,( v \circ \ga)')= 0 \qquad\hbox{ in $[0,1]$,  for   any $\ga \in \mathcal E$}
\end{equation}
\smallskip

We associate to \eqref{HJeik} the discrete equation  on $\VV$.

\begin{equation}\label{DFE}  \tag{{\bf DFE}}
V(x) = \min_{e \in - \EE_x} \big ( V(\oo(e))+ \si(e) \big )
\end{equation}
where $\si(e)= v_{\Psi(e)}(1)$, and  $v_{\Psi(e)}$  is the function appearing in Lemma \ref{previous} in relation with the equation $
H_{\Psi(e)} =0$.  We define
\[ \si(\xi) = \sum_{i=1}^M \si(e_i) \qquad\hbox{ for any path $\xi=(e_i)_{i=1}^M$.}\]

\begin{Proposition}\label{sottosuolo} A function $V : \VV \to \R$ is subsolution to \eqref{DFE} if and only if
\begin{equation}\label{trace}
V(y) - V(x) \leq \si(\xi) \qquad\hbox{for any path $\xi$ linking $x$ to $y$.}
\end{equation}
\end{Proposition}

\smallskip

There are results similar to Theorem \ref{theo:ponte}, Proposition \ref{pontebis} linking  \eqref{HJeik}   and \eqref{DFE}. We recall in
particular:

\begin{Proposition} The trace  on $\VV$ of any solution to \eqref{HJeik} is solution of \eqref{DFE}. Conversely, any solution of \eqref{DFE}
can be uniquely extended to a solution of \eqref{HJeik}.
\end{Proposition}

The Aubry set $\A$ is made up by vertices $y$ such that there is a cycle $\xi$  based on it with $\si(\xi)=0$.

\smallskip

In general  equation  \eqref{DFE} has  many solutions, not just differing by an additive constant. They are univocally determined, once a
trace satifying \eqref{trace}  is assigned on $\A$. The Aubry set plays in a sense the role of a hidden boundary.

\medskip

\subsection{Convergence results}

    We denote  by $u_\la$ , for $\la > 0$,  the solution to \eqref{HJ}.  and set $U_\la = u_\la|_{\VV}$. $U_\la$ is then  the solution of the
    corresponding discrete
equation \eqref{HJdis}.

\begin{Lemma}\label{lemconv} The functions  $u_\la:\Gamma \to \R$  are equibounded with respect to $\la >0$.
\end{Lemma}
\begin{proof} Let $v$ be a solution of the Eikonal equation on $\Gamma$. We can choose a large positive constant $\al$ such that  that $v +
\al$, $v-\al$ are super and subsolution of \eqref{HJ} for any $\la >0$.  We derive from Theorem \ref{compara}
\[ v - \al \leq  u_\la \leq   v + \al.\]
\end{proof}

\begin{Proposition}\label{conv} The functions $u_\la:\Gamma \to \R$ converge to a
solution of the Eikonal  equation on $\Gamma$, up to subsequences.
\end{Proposition}
\begin{proof}  We have that
\[\la \, u_\la(x) \geq \min \{  \la \, m, \,  - \max_\ga \, \max_s H_\ga(s,0) \} \qquad\hbox{for any   $x \in \Gamma$, $\la >0$,}\]
where $m$ is a lower bound for all the $u_\la$ as $x$ varies in $\Gamma$, see  Lemma \ref{lemconv}. We deduce, by the coercivity of the
$H_\ga$,\emph{} that
the functions $u_\la$ are equi--Lipschitz continuous and equibounded.
They are therefore convergent up to subsequences.

 Assume, to fix ideas, that $u_{\la_n}$, for some infinitesimal sequence $\la_n$, converges to a function $v$. Then $v \circ \ga$ is solution
 in $(0,1)$ of   \eqref{HJeikg},  for any arc $\ga$, by basic stability properties of viscosity solutions theory.

  Given a vertex $x$, there is, by the very definition of solution to \eqref{HJ}, an arc $\ga_n$ with $\ga_n(1)=x$
  such that $u_{\la_n}\circ \ga_n$ satisfies the state constraint boundary condition for \eqref{HJg},
  with $\la = \la_n$ at $s=1$. The arcs  being finite, we can extract a subsequence $\la_{n_k}$ of $\la_n$
  and select  $\ga$ with $\ga(1)=x$  such that  $u_{\la_{n_k}}\circ \ga$ satisfies
  the state constraint boundary condition for \eqref{HJg}, with $\la = \la_{n_k}$, for any $k$,  at $s=1$. By applying standard arguments, we
  derive that the limit function  $v \circ \ga$ satisfies the state constraint boundary condition  for \eqref{HJeikg}.  This concludes the
  proof, taking into   account the definition of solution to \eqref{HJeikg}.

\end{proof}

\begin{Proposition}\label{verano}  We have
\[\rho_\la(\al_n, \xi) \longrightarrow \al + \si(\xi) \qquad\hbox{as $\la \longrightarrow 0$,}\]
 for any path $\xi= (e_i)_{i = 1}^M$  and $ \al_n \longrightarrow \al \in \R$    with $\al_n \leq \un\al_\la(e_1)$ and
\[\rho_\la \left (\al_n, (e_i)_{i=1}^j \right ) \leq \un\al(e_{j+1}) \qquad\hbox{for $j=1,\cdots, M-1$, \, $n$ large.}\]
\end{Proposition}
\begin{proof}
The argument proceeds by induction on the length of $\xi$. If $M=1$ then the assertion is  a consequence of Lemma \ref{asy} . We assume it
true for any path of length less than or equal to $M-1$ and deduce it for the length $M$. We write $\ov \xi=(e_i)_{i=1}^{M-1}$ and use the
concatenation formula plus induction step, and Lemma \ref{asy} to get
\[ \lim_{\la \to 0}  \rho_\la(\al_n,\xi)= \rho_\la(\rho_\la(\al_n,\ov\xi),e_M)= \al+ \si(\ov\xi) +\si(e_M)= \al + \si(\xi).\]
\end{proof}

\smallskip

As pointed out in the Introduction, the next proposition should be compared with the convergence result for Mather sets obtained in
\cite{MaroSorrentino}.

\begin{Proposition}\label{superpoor}
The sets  $\A_\la$   are contained in  $\A$  for $\la$ sufficiently small.
\end{Proposition}
\begin{proof} The argument is by contradiction. Since the vertices are finite, we can therefore assume that there is $y \in \VV$  and $\la_n
\to 0$ with
\[  y \in \left ( \cap_n \A_{\la_n} \right ) \setminus \A.\]
 Taking into account the very definition of $\la$--Aubry set, Corollary \ref{corspring},  and the fact that the circuits are finite,  we have,
 up to extracting a subsequence from $\la_n$, that there exists a circuit $\xi = (e_i)_{i=1}^M$ based on $y$ satisfying  $U_{\la_n}(y)=
 \be_{\la_n}(\xi)$ for any $n$, and the conditions of Proposition \ref{spring}.  Taking into account Proposition \ref{presiste}, we then have
 \[ U_{\la_n}(y) = f_{\la_n}(y) \leq \un\al_{\la_n}(e_1)\]
 and
 \[\rho_{\la_n} \left (U_{\la_n}(y), (e_i)_{i=1}^j \right ) = U_{\la_n}(\oo(e_{j+1}))= f_{\la_n}(\oo(e_{j+1}))\leq \un\al_{\la_n}(e_{j+1})
 \qquad j=1,\cdots, M-1.\]
 Since the sequence $U_{\la_n}(y)$ is bounded by Lemma \ref{lemconv}, it is convergent to some $\al$, up to subsequences, and we have by
 applying Proposition \ref{verano}
 \[\al = \lim_n U_{\la_n}(y)= \rho_{\la_n}( U_{\la_n}(y),\xi)= \al + \si(\xi).\]
 This is impossible  because $ y \not\in \A$, and consequently by the very definition of $\A$,  $\si(\xi) > 0$.
\end{proof}

\smallskip

  We  consider the limit set  $\mathcal B$  defined as
\[ \mathcal B= \{y \in \VV \mid \exists \; \la_n \longrightarrow 0 \;\;\; \hbox{with} \;\; y \in \A_{\la_n}\}\]
It comes from Proposition \ref{superpoor} that $\mathcal B$ is contained in the Aubry set $\A$.
The next result shows that any limit of the $U_\la$ is uniquely determined by its trace on $\mathcal B$.

\smallskip

\begin{Proposition}\label{may}
Let $U_{\la_n}$ be a sequence of solution to \eqref{HJdis}  with $\la= \la_n$, converging to $V$. Then $V$ is a solution  of  \eqref{DFE}
satisfying
\[V(x)= \min \{V(y) + \si(\xi) \mid y \in \mathcal B, \, \xi \; \hbox{path joining $y$ to $x$}\}.\]
\end{Proposition}
\begin{proof}  We set to ease notations
\[U_n=U_{\la_n}, \quad \rho_n= \rho_{\la_n}, \quad \A_n= \A_{\la_n}.\]
We know from Proposition \ref{conv} that $V$ solves \eqref{DFE}.  For any $x \in \VV$, we have by Proposition \ref{prop:lambdaAubry} that
\[U_n(x)= \rho_n(U_n(y_n), \xi_n)\]
for some $y_n \in \A_n$, and some simple path $\xi_n$ linking $y_n$ to $x$. Since both vertices and simple paths are finite, we deduce that
there is a subsequence $\la_{n_k}$, $y \in \cap_k \A_{n_k} \subset \mathcal B$, a  simple path $\xi=(e_i)_{i=1}^M$ joining $y$ to $x$ such
that
\[U_{n_k}(x)= \rho_{n_k}(U_{n_k}(y), \xi) \qquad\hbox{for any $k$}\]
and in addition
\[U_{n_k}(\oo(e_j))= \rho_{n_k} \left (U_{n_k}(y), (e_i)_{i=1}^{j-1} \right )   \qquad\hbox{for any $j= 2, \cdots M-1$.}\]
Owing to Proposition \ref{presiste} and to  the inequality $f_{n_k} \geq U_{n_k}$,  we are  therefore  in the position to apply Proposition
\ref{verano}  and get
\[\lim_k  \rho_{n_k}(U_{n_k}(y), \xi) = V(y) + \si(\xi).\]
This implies
\[V(x)\geq  \min \{V(y) + \si(\xi) \mid y \in \mathcal B, \, \xi \; \hbox{path joining $y$ to $x$}\}.\]
The converse inequality is a consequence of $V$ being solution to \eqref{DFE}, see Proposition \ref{sottosuolo}.
\end{proof}

\bigskip

\begin{appendix}

\section{Graphs and networks} \label{immaterial}

An {\it immersed network} or {\it continuous graph}  is a subset $
\Gamma \subset \R^N$ of the form
\[ \Gamma = \bigcup_{\ga \in \EN} \, \gamma([0,1]) \subset \R^N,\]
where $\EN$ is a finite collection of regular simple curves, called
{\it arcs} of the network,  we assume for simplicity parameterized
in $[0,1]$.  The main condition is
\begin{equation}\label{netw}
    \ga((0,1)) \cap \ga'([0,1]) = \emptyset \qquad\hbox{whenever $\ga \neq
\pm \ga'$,}
\end{equation}
where for any  arc $\ga$,  the {\it inverse arc}
$- \ga$ defined as
\[- \ga(s)= \ga( 1 -s) \qquad\hbox{for $s \in [0,1]$.}\]
 We make precise that we consider throughout the paper $\ga$, $-\ga$ as distinct arcs. We call  {\it vertices}
initial  and terminal points of the arcs, and denote  by  $\VV$ the
sets of all such vertices. Note that \eqref{netw} implies that
\[\ga((0,1)) \cap \VV  = \emptyset \qquad\hbox{for any $\ga \in
\EN$.}\] We assume that the network  is  connected, namely given two
vertices there is a finite concatenation of  arcs linking them.

As already pointed out, we do not put any restriction on the geometry of the network.

\medskip

A  graph $\XX=(\VV ,\EE )$ is an
ordered pair of  sets $\VV$ and $\EE$, which are  called,
respectively, {\it vertices} and (directed) {\it edges}, plus two
functions:
$$\oo: \EE \longrightarrow \VV $$
which associates to each (oriented) edge its {\it origin} (initial
vertex), and
\begin{eqnarray*}
 - {\phantom{o}}: \EE &\longrightarrow& \EE \\
e &\longmapsto& - e,
\end{eqnarray*}
which  changes  orientation, and  is a
fixed point free involution.
We  define the terminal vertex of $e$ as
\[\tt  (e)= \oo ( - e)\]
We consider $e$ and $-e$ as distinct edges. We call {\it loop}  any edge $e$ with $\oo(e)=\tt(e)$. We define {\it path}  $\xi=(e_1, \cdots,
e_M)$ any finite sequence
of concatenated  edges, namely satisfying
\[\tt  (e_j)=\oo (e_{j+1}) \qquad\hbox{for any $j= 1, \cdots, M-1$.}\]
We define the {\em length of a path} as the number of its  edges.
We set $\oo (\xi)= \oo (e_1)$, $\tt  (\xi)= \tt  (e_M)$.
We call a path {\it closed} or a {\it cycle} if $\oo (\xi)= \tt
(\xi)$.

Given two paths $\xi$, $\eta$, we say that  $\xi$ is contained in
$\eta$, mathematically  $\xi \subset \eta$, if the edges  of $\xi$
make up a subset of the  edges of $\eta$.
If the condition $\tt(\xi) = \oo(\eta)$ holds true, we denote by
$\xi \cup \eta$ the path obtained via concatenation of $\xi$ and
$\eta$.

We call  {\it simple} a path without repetition of vertices,  except
possibly the initial and terminal vertex, in other terms
$\xi=(e_i)_{i=1}^M$ is simple if
\[\tt(e_i) = \tt(e_j) \, \Rightarrow i=j.\]

\smallskip

\begin{Remark}\label{semplicissimo}
There are finite many simple paths  in  a finite graph. In
fact their number   is estimated from above by that of the sum of
the  $k$--permutations of $|\EE|$  objects for $ 2 \leq k \leq
|\EE|$.
\end{Remark}
\smallskip

\begin{Proposition}\label{cy3}  A path is simple if and only there is
no simple  cycle properly contained in it.
\end{Proposition}

We define a  {\it circuit} to be a simple cycle.

 \medskip

Given $x \in   \VV$, we set
\begin{equation}\label{neww3}
 -\EE_x= \{e \in \EE \mid \tt (e) =x\}.
\end{equation}

\medskip

Starting from a network, a graph can be defined taking as vertices the same vertices of $\Gamma$ and as edges the elements of any abstract set
$\EE$ equipotent to $\mathcal E$.  We denote by  $\Psi$ a  bijection from $\EE$ to $\mathcal E$. The functions $\oo$,  $-$ yielding the graph
structure are given by
 \begin{eqnarray*}
   \oo(e)& = & \Psi(e)(0) \\
   - e &=& \Psi^{-1} (- {\Psi(e)}).
 \end{eqnarray*}

A graph corresponding to a connected network is connected in the sense that any two vertices are linked by
some path. \\
\bigskip

\medskip

\section{Basic material on HJ equations in $(0,1)$}\label{basicbis}

Given a  Lipschitz--continuous  function $w$ in $[0,1]$, we set for
$s \in [0,1]$
\begin{equation}\label{dbarra}
   \partial w(s) = \mathrm{co} \, \{p \mid p= \lim w'(s_i), \, w  \;\hbox{differentiable at } \; s_i, \,  \;s_i \to
s,\, s_i \in (0,1)\},
\end{equation}
where the symbol  co stands for convex hull.

\begin{Lemma}\label{techne}  Given a  Lipschitz--continuous  function $w$ in $[0,1]$,
 the function \; $s \mapsto  w(0) + q \, s $  \; is a constrained subtangent to $w$ at
$s=0$ if
\begin{equation}\label{techne01}
  q < \min \{p \mid p \in \partial w(0)\}.
\end{equation}
 the function \; $ s \mapsto  w(1) + q \,
    (s-1)$\;  is a constrained subtangent to $w$ at $s=1$ if
    \begin{equation}\label{techne02}
 q > \max \{p \mid p \in \partial w(1)\}.
\end{equation}
\end{Lemma}
\begin{proof}  We consider the case $s=1$, the assertion at $s=0$  can be proved similarly.
We assume condition \eqref{techne02}.   By the very definition of
$\partial w(1)$ there is an open interval $I$ containing $1$ with
\begin{equation}\label{techne1}
   q >  p \qquad\hbox{for any $s \in I \cap(0,1)$, $p \in \partial
w(s)$.}
\end{equation}
 Assume for purposes  of contradiction that there is $s \in  I
\cap(0,1)$ with
 \begin{equation}\label{techne2}
   w(s) < w(1) + q \,(s-1),
\end{equation}
 by Mean Value Theorem for generalized Clarke gradients (Theorem
2.3.7  in \cite{Clarke}), we find $\ov s  \in (s,1)
\subset I \cap (0,1)$ with
\[ w(s) - w(1)= \ov p \, (s-1) \qquad\hbox{for some $\ov p \in
\partial w(\ov s)$.}\]
We derive, in the light of \eqref{techne2}
\[q \, (s-1) > \ov p \, (s-1)\]
which in turn implies $ q < \ov p$, in contradiction with
\eqref{techne1}.
\end{proof}

\smallskip

\begin{proof} ( {\bf of Lemma \ref{chara}}) The function constantly equal to $c:=-  \frac 1\la \, \max_s  H_\ga(s,0)$ is a subsolution to
\eqref{HJg}. By the coercivity of $H_\ga$,  the family of subsolutions greater than or equal to $c$ is equi--Lipschitz continuous and is in
addition dominated by
 \[-  \frac 1\la \, \min \{ H_\ga(s,p) \mid s \in
[0,1],\; p \in \R\}.\]
 This shows that  $\um^\ga$ is finite valued
and    Lipschitz continuous. By standard arguments in viscosity solutions theory, the  maximality of $\um^\ga$   implies that it is a solution
in $(0,1)$, and satisfies the state constraints boundary condition at $s =0,\, 1$.

Assume now, for purposes of contradiction, that there is another
solution $w$ of the equation plus state constraints boundary
conditions. We set
\[ - \de = \min_{[0,1]} (w - \um^\ga) < 0.\]
We can use  suitable  sup--convolutions of $\um^\ga$ as test functions from below to prove that the minimizers of $w -\um^\ga$
cannot be interior points of the interval.   To show that they
cannot be boundary points,  we exploit Lemma  \ref{techne}. Assume,
to fix ideas, that $1$ is such a a minimizer. Therefore $\um^\ga-\de$ is
a constrained subtangent to $w$ at $1$. We set
\[p_0= \max \{p \mid p \in \partial \um^\ga(1)\},\]
 by  the definition of  $\partial\um^\ga$,  there is a sequence $s_i$  of differentiability points of $\um^\ga$ in $(0,1)$  converging to $1$
 with
\[ (\um^\ga) '(s_i) \longrightarrow p_0.\]
Since
\[\la(\um^\ga(s)-\de)+H_\ga(s, (\um^\ga)'(s))=-\la\de  \]
at any differentiability point $s$  of $\um^\ga$,  we derive  by the  continuity of $H_\ga$
\begin{equation}\label{chara03}
    \la \,(\um^\ga(1) -\de) + H_\ga(1,p_0)<0,
\end{equation}
and we can therefore find $q > p_0$
with
\begin{equation}\label{quella}
  \la w(1) + H_\ga(1,q)= \la \, (\um^\ga(1)- \de) + H_\ga(1,q) < 0.
\end{equation}
By Lemma \ref{techne} the function $s \mapsto (\um^\ga(1) - \de)  + q \, (s-1)$ is
constrained subtangent to $(\um^\ga - \de)$ at $1$ and consequently also to $w$
at $1$. Inequality \eqref{quella} shows that $w$ does not satisfy
the state constraint boundary condition  at $1$, reaching a
contradiction.\\

\end{proof}

\smallskip

\begin{proof}({\bf of Lemma \ref{wellposed}})
 The function  $v \equiv c$ with
\[c= \min \left\{- \frac 1 \la \, \max_{s \in[0,1]} H_\ga(s,0)\, ,\, \al
\right \},\]    is subsolution to
\eqref{HJg} taking a value less than or equal to $\al$ at $s=0$.  We deduce that
\[ u^\ga_\al(s)= \sup \{ v(s) \mid v \:\hbox{subsolutions to \eqref{HJg} with $v(0) \leq \al$, $v \geq c$}\}\]
and  by the coercivity of $H_\ga$  the functions of this family are  equi--Lipschitz continuous   and equibounded.  This proves that $u^\ga_\al$
is a Lipschitz continuous subsolution to \eqref{HJg}. The supersolution  property and the validity of the state constraint boundary condition
at $s=1$ are straightforward consequences of the maximality property.

If $\al \leq \um^\ga(0)$ then there is a subsolution taking the value $\al$ at $0$ and consequently by maximality $u^\ga_\al(0)=\al$.
Conversely, if $u^\ga_\al(0)= \al$ then $\um^\ga(0) \geq u^\ga_\al(0) = \al$.
\end{proof}

\smallskip

\begin{proof} ({\bf of Proposition \ref{fork}}) Let $u$ be a solution to \eqref{HJg} plus Dirichlet boundary conditions. The asserted
uniqueness comes from Theorem \ref{theo:comprinc} and \eqref{fork2} is a direct consequence of the definition of $u^\ga_\al$, $u^{-\ga}_\be$
 and Remark \ref{remfork}. Conversely, let us assume \eqref{fork2},   we
 define
\begin{eqnarray*}
 \ov u(s) &=& \min \{  u^\ga_\al(s),  u^{-\ga}_\be(1-s)\} \\
  \un u(s) &=& \max \{  u^{-\ga}_\be(1-s) + \al -  u^{-\ga}_\be(1) \, , \,  u^\ga_\al(s) + \be -
  u^\ga_\al(1)\}.
\end{eqnarray*}
The functions $\ov u$, $\un u$ are super and subsolutions to
\eqref{HJg}, respectively.   We derive from \eqref{fork2} and Remark
\ref{remfork} that
\[ \al\leq u^{-\ga}_\be(1) \leq \um^\ga(0) \]
 and so
$u^\ga_\al(0)= \al$ by Lemma \ref{wellposed}  and $\ov u(0)= \al$. We
also have by \eqref{fork2}
\[\al=u^\ga_\al(0) \geq u^\ga_\al(0) + \be -u^\ga_\al(1)\]
 which implies $\un u(0)= \al$.  Similarly
\[ \be \leq u^\ga_\al(1) \leq \um^\ga(1)\]
which implies $u^{-\ga}_\be (0) = \be$ and $\ov u(1)= \be$, in
addition
\[\be = u^{-\ga}_\be(0)  \geq u^{-\ga}_\be(0) +\al
-u^{-\ga}_\be(1)\] which gives $\un u(1)= \be$.
  This shows that $\un
 u$, $\ov u$ satisfy the same boundary Dirichlet conditions and are,
 in addition, both Lipschitz--continuous. Existence of the claimed
 solution then comes via a straightforward application of Perron
 Method, see \cite{Barles}.
\end{proof}

\smallskip

\begin{proof}{ \bf of Lemma \ref{asy}} We have  that
\[\la_n \, u_n(s)  \geq \min \{- \max_s H_\ga(s,0),\, \al -1 \}\qquad\hbox{for any $s \in [0,1]$, $n$ large.}\]
 This implies that the $u_n$ are equibounded and equi--Lipschitz continuous.
They  therefore converge, up to
subsequences, to some function $u$  with $u(0)= \al$. By stability
properties of viscosity solutions $u$ solves  \eqref{HJeikloc}.
Therefore
\begin{equation}\label{asy0}
   u \leq \al + v.
\end{equation}
If $a_\ga= 0$ then the above inequality must be an equality. If
instead  $ a_\ga < 0$ then there is a strict subsolution $w$ of $H_\ga=
0$ with
\begin{equation}\label{asy01}
  H_\ga(s,w') \leq - \de \quad\hbox{ for a suitable $\de >0$  \; and }  w(s) \leq  0,
\end{equation}
 We consider a sequence of positive numbers $\mu_k$ converging to
$1$ and  a subsequence
$\la_{n_k}$ of $\la_n$ with
\begin{equation}\label{asy2}
    \la_{n_k} \leq \frac{(1- \mu_k)\de}{\mu_k} \, \frac 1M,
\end{equation}
where $M$ is an upper bound of $\al_n + v(s)$ for  $n$ large and $s$ varying in
$[0,1]$. We exploit  \eqref{asy01},  \eqref{asy2} and the convex character of $H_\ga$  to
get
\begin{eqnarray*}
   && \la_{n_k} (\mu_k \, (\al_{n_k} + v) + (1- \mu_k) \, w) + H_\ga(s,\mu_k \, Dv + (1- \mu_k) \,D w) \\
  &\leq& \la_{n_k} \, \mu_k \, (\al_{n_k} + v)  - (1-\mu_k) \, \de  \leq \la_{n_k} \, \mu_k \, M  -
  (1-\mu_k) \, \de \\  &\leq& \frac{(1-\mu_k) \, \de}{\mu_k} \, \frac 1M \, \mu_k \, M   -(1-\mu_k) \, \de= 0 .
\end{eqnarray*}
We thus see  that $\mu_k \, (\al_{n_k} + v) + (1- \mu_k)
\, w$ is subsolution to \eqref{HJg} with $\la =  \la_{n_k}$ taking
in addition, by \eqref{asy01}, a value less than $\al_{n_k}$ at $s=
0$, at least for $k$ large.   We infer by the maximality property of
$u_{n_k}$
\[ u_{n_k} \geq \mu_k \,(\al_{n_k} + v) + (1- \mu_k) \, w \qquad\hbox{in
$[0,1]$,}\] so that
\[\liminf_k u_{n_k} \geq \lim_k \mu_k \, (\al_{n_k} + v) + (1- \mu_k) \, w =
(\al + v).\] The above relation, together with \eqref{asy0},
shows the assertion.

\end{proof}

\end{appendix}
\bigskip



\vspace{10 pt}

\begin{thebibliography}{10}
\expandafter\ifx\csname
natexlab\endcsname\relax\def\natexlab#1{#1}\fi
\expandafter\ifx\csname bibnamefont\endcsname\relax
  \def\bibnamefont#1{#1}\fi
\expandafter\ifx\csname bibfnamefont\endcsname\relax
  \def\bibfnamefont#1{#1}\fi
\expandafter\ifx\csname citenamefont\endcsname\relax
  \def\citenamefont#1{#1}\fi
\expandafter\ifx\csname url\endcsname\relax
  \def\url#1{\texttt{#1}}\fi
\expandafter\ifx\csname urlprefix\endcsname\relax\def\urlprefix{URL
}\fi \providecommand{\bibinfo}[2]{#2}
\providecommand{\eprint}[2][]{\url{#2}}





\bibitem{BardiCapuzzo}
Martino Bardi and Italo Capuzzo-Dolcetta.
\newblock Optimal control and viscosity solutions of Hamilton-Jacobi-Bellman equations.
\newblock {\em  Systems \& Control: Foundations \& Applications}, BirkhŠuser Boston, Inc., Boston, MA, xviii+570 pp., 1997.


\bibitem{Barles}
Guy Barles.
\newblock Solutions de viscosit\'e des \'equations de Hamilton-Jacobi. (French) [Viscosity solutions of Hamilton-Jacobi equations].
\newblock {Math\';ematiques \& Applications} 17, Springer-Verlag, Paris, x+194, 1994.



\bibitem{CamilliMarchiSchieborn}
Fabio Camilli, Claudio Marchi and Dirk Schieborn.
\newblock Eikonal equations on ramified spaces.
\newblock {\em Interfaces Free Bound.} 15 (1): 121--140, 2013.

\bibitem{Clarke}
 Frank H. Clarke.
  \newblock  Optimization and Nonsmooth Analysis.
  \newblock {Society for Industrial and  Applied Mathematics}, Philadelphia, 1990.


  \bibitem{DFIZ}
Andrea Davini, Albert Fathi, Renato Iturriaga and Maxime
Zavidovique.
\newblock Convergence of the solutions of the discounted equation: the discrete case.
\newblock {\em  Math. Z.}284 (3-4): 1021--1034, 2016.



\bibitem{Fathi}
Albert Fathi.
\newblock Weak KAM in Lagrangian Dynamics.
\newblock Lecture notes, 2008.


\bibitem{FathiSiconolfi}
Albert Fathi and Antonio Siconolfi.
\newblock  PDE aspects of Aubry-Mather theory for quasiconvex Hamiltonians.
\newblock {\em Calc. Var. Partial Differential Equations} 22 (2): 185--228, 2005.

\bibitem{LionsSouganidis}
Pierre-Louis Lions and Panagiotis Souganidis.
\newblock Viscosity solutions for junctions: well posedness and stability.
\newblock {\em Atti Accad. Naz. Lincei Rend. Lincei Mat. Appl.} 27 (4): 535--545, 2016.

\bibitem{MaroSorrentino}
 Stefano Mar\'o and Alfonso Sorrentino.
 \newblock  Aubry-Mather theory for conformally symplectic systems.
 \newblock{ \em Commun. Math. Phys.} 354 : 775–-808, 2017.



\bibitem{CamilliSchieborn}
  Dirk Schieborn and Fabio Camilli.
\newblock Viscosity solutions of Eikonal equations on topological networks.
\newblock {\em  Calc. Var. Partial Differential Equations} 46 (3-4):671--686, 2013.

\bibitem{SiconolfiSorrentino}
Antonio Siconolfi and  Alfonso Sorrentino
\newblock{ Global Results for Eikonal Hamilton-Jacobi Equations on Networks}
\newblock{ \em Analysis and PDE} 11:  171-–211, 2018.



\bibitem{Sunada}
Toshikazu Sunada.
\newblock Topological crystallography. With a view towards discrete geometric analysis.
\newblock {\em Surveys and Tutorials in the Applied Mathematical Sciences}, 6. Springer, Tokyo, xii+229 pp., 2013.

\bibitem{WangYan}
Kaizhi Wang, Lin Wang and Jun Yan.
\newblock Aubry-Mather and weak KAM theories for contact Hamiltonian systems. Part 1: Strictly increasing case.
\newblock{Preprint 2018}  arXiv 1801.05612v4

\end{thebibliography}
\end{document}